\documentclass[11pt,english]{amsart}
\usepackage[T1]{fontenc}
\usepackage[latin1]{inputenc}
\usepackage{verbatim}
\usepackage{amstext}
\usepackage{amsthm}
\usepackage{amssymb}

\makeatletter

\providecommand{\tabularnewline}{\\}

\numberwithin{equation}{section}
\numberwithin{figure}{section}
\theoremstyle{plain}
\newtheorem{thm}{\protect\theoremname}
\theoremstyle{plain}
\newtheorem{prop}[thm]{\protect\propositionname}
\theoremstyle{plain}
\newtheorem{cor}[thm]{\protect\corollaryname}
\theoremstyle{plain}
\newtheorem{lem}[thm]{\protect\lemmaname}
\theoremstyle{remark}
\newtheorem{rem}[thm]{\protect\remarkname}

\usepackage{tikz}
\usepackage{pgfplots}

\usepackage{babel}
\providecommand{\corollaryname}{Corollary}
\providecommand{\lemmaname}{Lemma}
\providecommand{\propositionname}{Proposition}
\providecommand{\remarkname}{Remark}
\providecommand{\theoremname}{Theorem}

\makeatother

\usepackage{babel}
\providecommand{\corollaryname}{Corollary}
\providecommand{\lemmaname}{Lemma}
\providecommand{\propositionname}{Proposition}
\providecommand{\remarkname}{Remark}
\providecommand{\theoremname}{Theorem}

\begin{document}
%\addtolength{\parskip}{20pt}
\addtolength{\textwidth}{0mm} \addtolength{\hoffset}{0mm} \addtolength{\textheight}{5mm}
\addtolength{\voffset}{-5mm}

%\subjclass{Primary: 14J29} ; Secondary: 14G10, 14G15}

\global\long\def\CC{\mathbb{C}}%
\global\long\def\BB{\mathbb{B}}%
\global\long\def\PP{\mathbb{P}}%
\global\long\def\QQ{\mathbb{Q}}%
\global\long\def\RR{\mathbb{R}}%
\global\long\def\FF{\mathbb{F}}%

\global\long\def\DD{\mathbb{D}}%
\global\long\def\NN{\mathbb{N}}%
\global\long\def\ZZ{\mathbb{Z}}%
\global\long\def\HH{\mathbb{H}}%
\global\long\def\Gal{{\rm Gal}}%
\global\long\def\bA{\mathbf{A}}%
\global\long\def\kP{\mathfrak{P}}%
\global\long\def\kQ{\mathfrak{q}}%
\global\long\def\ka{\mathfrak{a}}%
\global\long\def\kP{\mathfrak{p}}%
\global\long\def\cO{\mathcal{O}}%
\global\long\def\cC{\mathfrak{\mathcal{C}}}%
\global\long\def\cM{\mathcal{M}}%
\global\long\def\cK{\mathcal{K}}%
\global\long\def\cN{\mathcal{N}}%
\global\long\def\cP{\mathcal{P}}%
\global\long\def\cH{\mathcal{H}}%
\global\long\def\a{\alpha}%
\global\long\def\b{\beta}%
\global\long\def\d{\delta}%
\global\long\def\D{\Delta}%
\global\long\def\L{\Lambda}%
\global\long\def\g{\gamma}%
\global\long\def\G{\Gamma}%
\global\long\def\d{\delta}%
\global\long\def\D{\Delta}%
\global\long\def\e{\varepsilon}%
\global\long\def\k{\kappa}%
\global\long\def\l{\lambda}%
\global\long\def\m{\mu}%
\global\long\def\o{\omega}%
\global\long\def\p{\pi}%
\global\long\def\P{\Pi}%
\global\long\def\s{\sigma}%
\global\long\def\S{\Sigma}%
\global\long\def\t{\theta}%
\global\long\def\T{\Theta}%
\global\long\def\f{\varphi}%

\global\long\def\deg{{\rm deg}}%
\global\long\def\det{{\rm det}}%

\global\long\def\Dem{Proof: }%
\global\long\def\ker{{\rm Ker}}%
\global\long\def\im{{\rm Im}}%
\global\long\def\rk{{\rm rk}}%
\global\long\def\car{{\rm car}}%
\global\long\def\fix{{\rm Fix( }}%
\global\long\def\card{{\rm Card }}%
\global\long\def\codim{{\rm codim}}%
\global\long\def\coker{{\rm Coker}}%

\global\long\def\pgcd{{\rm pgcd}}%
\global\long\def\ppcm{{\rm ppcm}}%
\global\long\def\la{\langle}%
\global\long\def\ra{\rangle}%

\global\long\def\Alb{{\rm Alb}}%
\global\long\def\Jac{{\rm Jac}}%
\global\long\def\Disc{{\rm Disc}}%
\global\long\def\Tr{{\rm Tr}}%
\global\long\def\Nr{{\rm Nr}}%

\global\long\def\NS{{\rm NS}}%
\global\long\def\Pic{{\rm Pic}}%

\global\long\def\Km{{\rm Km}}%
\global\long\def\rk{{\rm rk}}%
\global\long\def\Hom{{\rm Hom}}%
\global\long\def\End{{\rm End}}%
\global\long\def\aut{{\rm Aut}}%
\global\long\def\SSm{{\rm S}}%
\global\long\def\psl{{\rm PSL}}%
\global\long\def\cu{{\rm (-2)}}%
\global\long\def\mod{{\rm \,mod\,}}%

\title[Generalized Kummer surfaces and Fourier--Mukai partners]{Fourier--Mukai partners and generalized Kummer structures on generalized
Kummer surfaces of order $3$}
\author{Xavier Roulleau, Alessandra Sarti}
\begin{abstract}
A generalized Kummer surface $X$ of order $3$ is the minimal resolution
of the quotient of an abelian surface $A$ by an order $3$ symplectic
automorphism. We study a generalization of a problem of Shioda for
classical Kummer surfaces, which is to understand to what extend $X$
is determined by $A$ and conversely. The surface $X$ posses a big
and nef divisor $L_{X}$ such that $L_{X}^{2}=0$ or $2$ mod $6$;
we suppose it has Picard number $19$. We show that for surfaces with
$L_{X}^{2}=6k$ with $k\neq0,6\mod9$, the surface $X$ determines
the transcendental lattice $T(A)$ of $A$ and the Hodge structure
on $T(A)$. Conversely if $A$ and $B$ are Fourier-Mukai partners
(i.e. if the Hodge structures of their transcendental lattices are
isomorphic) and $Y$ is the generalized Kummer surface which is the
minimal resolution of the quotient of $B$ by an order $3$ symplectic
automorphism, we obtain that $X$ and $Y$ are isomorphic. These results
are also known to hold for surfaces with $L_{X}^{2}=2\mod6$ from
a previous work \cite{RS}. When $k=0\text{ or }6\mod9,$ we show
that $X$ determines $T(A)$ and its Hodge structure, but the converse
does not hold in general. 
\end{abstract}

\maketitle

\section{Introduction}

An (algebraic and complex) generalized Kummer surface $X=\Km_{3}(A)$
of order $3$ is a K3 surface which is the minimal resolution of the
quotient $A/G_{A}$ of an abelian surface $A$ by an order $3$ symplectic
automorphism group $G_{A}$. In \cite{RS,KRS,Roulleau} we study when
there exists another Abelian surface $B$ and order $3$ automorphism
group $G_{B}$ such that $(B,G_{B})$ is not isomorphic to $(A,G_{A})$
but the generalized Kummer surfaces $\Km_{3}(A)$ and $\Km_{3}(B)$
are. That question is a natural extension to generalized Kummer surfaces
of a problem of Shioda on classical Kummer surfaces.

Recall that the quotient surface $A/G_{A}$ has $9$ cups singularities,
each of these being resolved by two $\cu$-curves on $X$. For a very
general abelian surface $A$, the K3 surface $X$ has Picard number
$19$. In the present paper, we will always suppose that the surface
$X$ has Picard number $19$. The positive generator $L_{X}$ of the
orthogonal complement in the Néron-Severi group $\NS(X)$ of the $18$
exceptional curves of the resolution $X\to A/G_{A}$ satisfies $L_{X}^{2}=0$
or $2\mod6$.

Let us also recall that for an abelian surface (respectively a K3
surface) $Y$, a Fourier-Mukai partner of $Y$ is an abelian surface
(respectively K3 surface) $Y'$ such that there is an isomorphism
of Hodge structures 
\[
(T(Y),\CC\o_{Y})\simeq(T(Y'),\CC\o_{Y'}).
\]
Here the transcendental lattice $T(Y)$ is the orthogonal complement
in $H^{2}(Y,\ZZ)$ of the Néron-Severi lattice $\NS(Y)\subset H^{2}(Y,\ZZ)$
and $\o_{Y}$ is a generator of the space of holomorphic $2$-forms.

Let $(A,G_{A})$ and $(B,G_{B})$ be two abelian surfaces with an
order $3$ symplectic automorphism group. Suppose that $X=\Km_{3}(A)$
is a generalized Kummer surface with Picard number $19$. Consider
the following two assertions:\\
 (I) The surface $B$ is a Fourier--Mukai partner of $A$. \\
 (II) The surfaces $\Km_{3}(B)$ and $\Km_{3}(A)$ are isomorphic.
\\
 As shown in \cite{RS}, for generalized Kummer surfaces $X$ with
$L_{X}^{2}=2\mod6$, the assertions (I) and (II) are equivalent. The
aim of this paper is to prove the following result: 
\begin{thm}
\label{thm:Main1}Let $X$ be a generalized Kummer surface with Picard
number $19$ and such that $L_{X}^{2}=6k$, for $k\in\NN^{*}$.\\
 a) Suppose that $k\neq0\text{ and }6\mod9$. Then (I) is equivalent
to (II).\\
 b) Suppose $k=0$ or $6\mod9$. Then (II) implies (I), but (I) does
not implies (II) in general: there exist abelian surfaces with order
$3$ symplectic automorphisms that are Fourier-Mukai partners and
such that the associated generalized Kummer surfaces are not isomorphic. 
\end{thm}

By \cite{HLOY2}, since $X=\Km_{3}(A)$ has Picard number $19>2+\ell$
(here $\ell$ is the length of the discriminant group of $\NS(X)$,
which is also the length of $T(X)$, and therefore is $\leq3$), if
a K3 surface is a Fourier-Mukai partner of $\Km_{3}(A)$, then it
is isomorphic to $\Km_{3}(A)$. The assertion (II) is therefore equivalent
to :\\
 (II') The surfaces $\Km_{3}(A)$ and $\Km_{3}(B)$ are Fourier-Mukai
partners, i.e. there exists a Hodge isometry 
\[
(T(\Km_{3}(A)),\CC\omega_{\Km_{3}(A)})\simeq(T(\Km_{3}(B)),\CC\omega_{\Km_{3}(B)}).
\]

In case $L_{X}^{2}=2\mod3$, the equivalence between (I) and (II')
(and therefore with (II)) is much simpler. Indeed, it is easy to compare
the Hodge structures, since in that case $T(\Km_{3}(A))$ is isometric
to $T(A)(3)$, (where the $3$ means that the quadratic intersection
form of the lattice is multiplied by $3$). In case $L_{X}^{2}=0\mod6$,
one only knows that $T(\Km_{3}(A))$ contains an index $3$ sub-lattice
isometric to $T(A)(3)$ and that subtlety makes the proof much more
involved.

A generalized Kummer structure on the generalized Kummer surface $X$
is an isomorphism class of pairs $(B,G_{B})$ such that the associated
generalized Kummer surface $\Km_{3}(B)$ is isomorphic to $X=\Km_{3}(A)$.
The generalized problem of Shioda can be rephrased as the problem
to understand if there is a unique generalized Kummer structure. Theorem
\ref{thm:Main1} shows that a generalized Kummer structure $\{(B,G_{B})\}$
on $X$ is such that $B$ is a Fourier-Mukai partner of $A$, for
$X$ such that $L_{X}^{2}\neq0,6\mod9$. Theorem \ref{thm:Main1}
is a key-result in order to compute the number of generalized Kummer
structures in \cite{Roulleau} according to the value of $L_{X}^{2}$,
for $L_{X}^{2}\neq0,6\mod9$.

In \cite{HLOY,HLOY2}, Hosono, Lian, Oguiso and Yau study the analogous
problem for the classical Kummer surfaces (and compute the number
of Kummer structures of some Kummer surfaces). For these surfaces,
one always has $T(\Km(A))=T(A)(2)$, and the equivalence between the
analog of (I) and (II) for Kummer surfaces is immediate. The fact
that for generalized Kummer surfaces (I) does not imply (II) is all
the more noteworthy.

The paper is structured as follows: In Section 2, we give preliminaries
and notations on lattice theory. In Section 3, we explain how we proceed
to prove Theorem \ref{thm:Main1}. One wants in particular to know
the over-lattices of $T(A)(3)$ which are isomorphic to $T(X)$. Sections
4 and 5 are devoted to compute these over-lattices. Section 6 is a
proof of Theorem \ref{thm:Main1} when $k\neq0,6\mod9$, and section
7 deals with the remaining cases.

\textbf{Acknowledgements.} The authors thank Simon Brandhorst and
Igor Reider for useful conversations. X.R. acknowledges support from
Centre Henri Lebesgue ANR-11-LABX-0020-01.

\section{Preliminaries on lattices and notations }

The following section gives preliminaries and notations on lattice
theory, standard references are \cite[Section 1]{Nikulin} and \cite{CS}.
Sections \ref{subsec:The-discriminant-group} and \ref{subsec:Examples-of-quadratic},
on the torsion quadratic modules of the form $(\frac{u}{v})$, are
well-known results, but recalling it here fixes the notations and
gives the level of details we need. In Section \ref{subsec:Genus-of-lattice,},
we recall some lattice-theoretic results of Nikulin.

\subsection{\label{subsec:The-discriminant-group}The discriminant group of a
lattice as a torsion quadratic module}

Let $L$ be an even lattice ; the intersection of two elements $v,v'\in L$
is denoted by $vv'\in\ZZ$. The bilinear pairing extends to $L\otimes\QQ$,
and the dual of $L$ is 
\[
\check{L}:=\{v\in L\otimes\QQ\,|\,\forall w\in L,\,vw\in\ZZ\}.
\]
The discriminant group of $L$ is a torsion quadratic module, denoted
by $\text{A}_{L}$: its subjacent group is the finite abelian group
$\check{L}/L$, and its quadratic form $q_{L}:\check{L}/L\to\QQ/2\ZZ$
is defined for $\frac{1}{n}w\in\check{L}$ (for $n\in\ZZ,n\neq0$
and $w\in L$ such that $\frac{1}{n}w\in\check{L}$) by 
\[
q_{L}(\frac{1}{n}w)=\frac{1}{n^{2}}w^{2}\mod2\ZZ.
\]
If $v\in L$, one has 
\[
(\frac{1}{n}w+v)^{2}=\frac{1}{n^{2}}w^{2}+2\frac{1}{n}wv+v^{2}.
\]
Since $\frac{1}{n}w\in\check{L}$, $2\frac{1}{n}wv$ is an even integer
and since $L$ is even, one has $v^{2}\in2\ZZ$, thus the quadratic
form $q_{L}$ is a well defined function on the finite group $\check{L}/L$.

Let $L$ be the lattice $\ZZ^{n}$ with Gram matrix $G$ for the canonical
basis. The columns $c_{1},\dots,c_{n}$ of $G^{-1}$ considered as
elements of $\QQ^{n}/\ZZ^{n}$ generate the group $\text{A}_{L}$,
moreover the quadratic form $q_{L}$ on these generators is given
by the matrix $G^{-1}=(^{t}c_{i}Gc_{j})_{1\leq i,j\leq n}$, where
the diagonal entries of $G^{-1}$ are taken modulo $2\ZZ$, and the
other entries are taken modulo $\ZZ$ (and respecting the symmetry).

\subsection{\label{subsec:Genus-of-lattice,}Genus of lattice, discriminant group
and over-lattices}

By definition, two lattices $L_{1},L_{2}$ are in the same genus if
$L_{1}\otimes\QQ_{p}\simeq L_{2}\otimes\QQ_{p}$ for all primes $p$
and $L_{1}\otimes\RR\simeq L_{2}\otimes\RR$. One has 
\begin{thm}
(\cite[Corollary 1.9.4]{Nikulin}) The lattices $L_{1},L_{2}$ are
in the same genus if and only if they have the same signature and
have isometric discriminant groups. 
\end{thm}

Any over-lattice $L_{H}$ of $L$ is (up to isometry) the pull-back
$L_{H}=\pi^{-1}(H)$ by the quotient map $\pi:\check{L}\to\check{L}/L$
of an isotropic subgroup $H$ contained in $\text{A}_{L}=\check{L}/L$.

An isometry $g\in O(L)$ induces an isometry $\bar{g}$ of $\text{A}_{L}$.
If $H_{1}\subset\text{A}_{L}$ is an isotropic sub-group, so is $H_{2}=\bar{g}(H_{1})$
and $g$ induces an isometry between the over-lattices $L_{H_{1}}$
and $L_{H_{2}}$ preserving the lattice $L$. More generally, two
over-lattices $\iota_{1}:L\hookrightarrow L_{1},\iota_{2}:L\hookrightarrow L_{2}$
are said isomorphic if there exists an isometry $g$ of $L$ extending
to an isometry $\tilde{g}$ between $L_{1}$ and $L_{2}$. Then the
following diagram 
\[
\begin{array}{ccc}
L & \stackrel{\iota_{1}}{\to} & L_{1}\\
\downarrow g &  & \downarrow\tilde{g}\\
L & \stackrel{\iota_{2}}{\to} & L_{2}
\end{array}
\]
is commutative. Let $H_{1},H_{2}$ be the isotropic sub-groups of
$\text{A}_{L}$ corresponding to over-lattices $L_{1},L_{2}$. One
has: 
\begin{prop}
\label{prop:().-Nikulin1Over}(\cite[Proposition 1.4.3]{Nikulin}).
The over-lattices $\iota_{1}:L\hookrightarrow L_{1},\iota_{2}:L\hookrightarrow L_{2}$
are isomorphic if and only if there exist an isometry $g$ of $L$
such $\bar{g}(H_{1})=H_{2}$. 
\end{prop}

Let $M$ be a lattice containing $L$ with finite index, let $H_{1},\dots,H_{n}$
be the set of all isotropic sub-groups of the (finite) group $\text{A}_{L}$
such that their associated over-lattices $L_{1},\dots,L_{n}$ are
isometric to $M$. Let $\iota_{j}:L\hookrightarrow L_{j}$ be the
inclusion of $L$ in $L_{j}$. 
\begin{cor}
\label{cor:Unique-embedd-up-to-isom}Under the setting above, suppose
that the over-lattices $\iota_{j}:L\hookrightarrow L_{j}$, $j=1,\dots,n$
are isomorphic. Then there exist a unique embedding of $L$ in $M$
up to isometries of $L$ and $M$. 
\end{cor}

\begin{proof}
Let $\iota:L\hookrightarrow M$, $\iota':L\hookrightarrow M$ be two
embeddings, let $H,H'$ be the two isotropic sub-groups of $L$ corresponding
to these over-lattices. There exist two indices $s,t\in\{1,\dots,n\}$
such that (up to isometries) $\iota:L\hookrightarrow M$ is isomorphic
to $\iota_{s}:L\hookrightarrow L_{s}$ and $\iota':L\hookrightarrow M$
is isomorphic to $\iota_{t}:L\hookrightarrow L_{t}$, the over-lattices
$L_{s},L_{t}$ corresponding to $H_{s}$ and $H_{t}$ respectively.
Since by hypothesis $\iota_{s}:L\hookrightarrow L_{s}$ and $\iota_{t}:L\hookrightarrow L_{t}$
are isomorphic, the over-lattices $\iota:L\hookrightarrow M$, $\iota':L\hookrightarrow M$
are isomorphic, thus there exist isometries $g$ of $L$ and $\tilde{g}$
of $M$ such that $\tilde{g}\circ\iota=g\circ\iota'$. 
\end{proof}

\subsection{\label{subsec:Examples-of-quadratic}Examples of quadratic torsion
modules }

In this section, we give some examples of quadratic torsion modules
which we will later use to compute the discriminant groups of various
lattices.

Consider $\frac{u}{v}\in\QQ\setminus\{0\}$ with $u,v$ coprime such
that either $u$ or $v$ is even. Let us denote by $(\frac{u}{v})$
the torsion quadratic module $(\ZZ/v\ZZ,q)$ with quadratic form $q:\ZZ/v\ZZ\to\QQ/2\ZZ$
defined by 
\[
q(x)=\frac{ux^{2}}{v}\in(\frac{1}{v}\ZZ)/2\ZZ\subset\QQ/2\ZZ.
\]
If $x'=x+tv$ with $t\in\ZZ$, one has 
\[
q(x')=\frac{u(x^{2}+2xtv+t^{2}v^{2})}{v}=\frac{ux^{2}}{v}+u(2xt+t^{2}v)=\frac{ux^{2}}{v}=q(x),
\]
thus $q$ is well-defined; here we use that $u$ or $v$ is even at
the third equality (if both were odd, the form $q$ would be well-defined
only modulo $\ZZ$). For $u,u'$ coprime to $v$, one has 
\[
(\frac{u}{v})=(\frac{u'}{v})
\]
if and only if $u=u'\mod2v$.

Suppose that $v=ab$ with $a,b$ coprime integers, and let $s,t\in\ZZ$
such that $as+bt=1$. If $(s',t')$ is another solution, one has $s'=s+mb$
and $t'=t-ma$ for $m\in\ZZ$.

$\bullet$ When $u$ is even, (thus $a$ and $b$ are odd). Then,
since $u$ is even, one has $us=us'\mod2b$, $ut=ut'\mod2a$ and the
torsion quadratic modules 
\[
(\frac{us}{b}),\,\,(\frac{ut}{a})
\]
are independent of the choice of the solution $(s,t)$.

$\bullet$ When $u$ is odd and $a$, say, is even, $b$ is odd. The
relation $as+bt=1$ implies that $t$ is odd. Moreover, since $b$
is odd, up to changing $s$ by $s+mb$ for $m\in\ZZ$, one can choose
$s$ to be even, and in fact one must do it in order for the quadratic
form of $(\frac{us}{b})$ to have values in $\ZZ/2\ZZ$. That choice
of $s$ is unique modulo $2b$. The choice of $t$ is then unique
modulo $2a$, and in the notations: $(\frac{us}{b}),(\frac{ut}{a})$,
it is always implicitly supposed that $s$ is even.

In both cases $u$ even or odd, the relation 
\[
\frac{u}{ab}=\frac{tu}{a}+\frac{su}{b}
\]
implies that $(\frac{u}{v})$ decomposes as a direct sum 
\[
(\frac{u}{v})=(\frac{tu}{a})+(\frac{su}{b})
\]
where we use the canonical isomorphism $x\mod ab\to(x\,\mod a,x\mod b)$
between $\ZZ/ab\ZZ$ and $\ZZ/a\ZZ\oplus\ZZ/b\ZZ$ (we recall that
$a,b$ are coprime integers). For $v=ab$ with $ab$ coprime, we denote
by $(\frac{u}{v})_{a}=(\frac{tu}{a})$ the restriction of $(\frac{u}{v})$
to $\ZZ/a\ZZ$. Since $s$ is even, in $\ZZ/2a\ZZ$ one has $\bar{b}\bar{t}=\bar{1}$,
thus we sometimes write $(\frac{u}{v})_{a}=(\frac{u/b}{a})$.

The decomposition $(\frac{u}{ab})=(\frac{tu}{a})+(\frac{su}{b})$
implies that if $v=p_{1}^{n_{1}}\dots p_{k}^{n_{k}}$ is the factorization
of $v$ into a product of distinct primes, the torsion quadratic module
$(\frac{u}{v})$ is isometric to a direct sum 
\[
(\frac{u_{1}}{p_{1}^{n_{1}}})+\dots+(\frac{u_{k}}{p_{k}^{n_{k}}})
\]
for some integers $u_{i}$ coprime to $p_{i}$ and such that $u_{i}$
or $p_{i}$ is even. More generally, using that a finite abelian group
is the direct sum of its $p$-sub-groups, a torsion quadratic module
$M$ is a direct sum $M=\oplus_{p\,\text{prime}}M_{p}$ of $p$-sub-groups.
For $p\geq3$, $M_{p}$ is a direct sum of torsion quadratic modules
of the form $(\frac{u}{p^{n}})$ (for $u$ even, coprime to $p$ and
$n\in\NN^{*}$). For $p=2$, the situation is more complicated, see
\cite{CS}.

Let be $w\in\ZZ$ such that $\bar{w}\in\ZZ/v\ZZ$ is a unit. For any
$n\in\ZZ$, one has 
\[
u(w+nv)^{2}=uw^{2}+2nuwv+un^{2}v^{2}=uw^{2}\mod2v,
\]
(where in the last equality we use that $u$ or $v$ is even), thus
group of units $(\ZZ/v\ZZ)^{*}$ acts on the set of torsion quadratic
modules $(\frac{u}{v})$ by $\bar{w}.(\frac{u}{v})=(\frac{uw^{2}}{v})$.
One has 
\[
\forall x\in\ZZ/v\ZZ,\,\,\frac{uw^{2}}{v}x^{2}=\frac{u}{v}(\bar{w}x)^{2},
\]
therefore the automorphism of $\ZZ/v\ZZ$ defined by $x\to\bar{w}x$
is an isometry between the torsion quadratic modules $\bar{w}.(\frac{u}{v})$
and $(\frac{u}{v})$. Conversely, if there is an isometry between
$(\frac{u}{v})$ and $(\frac{u'}{v})$, the subjacent automorphism
$\psi$ of $\ZZ/v\ZZ$ being of the form $x\to\bar{w}x$ for some
$\bar{w}\in(\ZZ/v\ZZ)^{*}$, one has $u'=uw^{2}\mod2v$, thus $u'=uw^{2}\mod v$
and $u/u'\in(\ZZ/v\ZZ)^{*}$ is a square.

Suppose that $v$ is an odd prime. The set $S_{q}$ of square elements
in $(\ZZ/v\ZZ)^{*}$ is a sub-group of index $2$. Let $m_{2}:\ZZ/v\ZZ\to2\ZZ/\ZZ2v\ZZ$
be the map defined by $x+v\ZZ\to2x+2v\ZZ$; it is a bijection. Let
$E$ be the set 
\[
E=m_{2}((\ZZ/v\ZZ)^{*}).
\]
The set $E$ is in bijection with the set of quadratic torsion modules
by the map $u\in E\to(\frac{u}{v})$. There is a natural faithful
action of $S_{q}$ on $E$ given by 
\[
(s,u)\in S_{q}\times E\to m_{2}(m_{2}^{-1}(u)s)\in E.
\]
That action has two orbits, corresponding to the invertible square,
and invertible non-square in $S_{q}$. Thus there are only two isometry
classes among the torsion quadratic modules $(\frac{u}{v})$, $u\in E$.
For example, the only torsion quadratic modules on $\ZZ/3^{a}\ZZ$
up to isometry are $(\frac{2}{3^{a}})$ and $(\frac{4}{3^{a}})$ (for
$a\geq1$).

\section{Fourier-Mukai partners in case $L_{X}^{2}=0\protect\mod6$}

Suppose that the generalized Kummer surface $X=\Km_{3}(A)$ satisfies
$L_{X}^{2}=6k$. Then (see \cite[Theorem 7]{Roulleau}), there exists
a basis of $T(A)$ with Gram matrix 
\[
\left(\begin{array}{ccc}
-2k & 0 & 0\\
0 & 2 & 3\\
0 & 3 & 6
\end{array}\right),
\]
therefore $T(A)(3)$ has a basis with Gram matrix 
\begin{equation}
\left(\begin{array}{ccc}
-6k & 0 & 0\\
0 & 6 & 9\\
0 & 9 & 18
\end{array}\right).\label{eq:nom1}
\end{equation}
Let $T(X)$ be the transcendental lattice of $X$: this is the orthogonal
complement of $\NS(X)$ in $H^{2}(X,\ZZ)$. 
\begin{prop}
\label{prop:T(X)-unique-in-genus}The lattice $T(X)$ has a basis
with Gram matrix 
\[
\left(\begin{array}{ccc}
-6k & 0 & 0\\
0 & 6 & 3\\
0 & 3 & 2
\end{array}\right).
\]
It is unique in its genus. The discriminant group $\text{A}_{T(X)}$
is generated by $(\frac{1}{6k},0,0),\,\,(0,\frac{2}{3},-1)$ and is
isomorphic to $\ZZ/6k\ZZ\times\ZZ/3\ZZ$. 
\end{prop}

\begin{proof}
See \cite[Section 2.4]{Roulleau} for the first two assertions; the
last one is clear by the shape of the Gram matrix. 
\end{proof}
One has 
\begin{prop}
There exists a morphism $\pi_{A*}:T(A)(3)\to T(X)$ such that $\pi_{A*}$
is an isometry onto its image and $\pi_{A*}(T(A)(3))$ has index $3$
in $T(X)$. 
\end{prop}

\begin{proof}
The morphism $\pi_{A*}$ is defined in \cite[Section 2.3]{RS}, it
is induced by the natural rational map $\pi_{1}:A\dashrightarrow\Km_{2}(A)=X$,
and it is shown that the lattice $\pi_{A*}(T(A))$ is isometric to
$T(A)(3)$. For the assertion on the index, it is sufficient to compare
the discriminant of the two lattices. 
\end{proof}
Suppose that there exists in $T(A)(3)\otimes\QQ$ a unique over-lattice
$T$ of $T(A)(3)$ such that $T$ is isometric to $T(X)$, where $X=\Km_{3}(A)$.
Let us suppose moreover that $T$ contains a unique sub-lattice isometric
to $T(A)(3)$. Let $(B,G_{B})$ be another abelian surface $B$ with
an order $3$ symplectic automorphism group $G_{B}$. 
\begin{thm}
\label{thm:Main3}Under the above hypothesis, there exists a Hodge
isometry 
\begin{equation}
(T(A),\CC\o_{A})\simeq(T(B),\CC\o_{B})\label{eq:hodgeIso-1-1}
\end{equation}
if and only if there exists a Hodge isometry 
\[
(T(\Km_{3}(A)),\CC\omega_{\Km_{3}(A)})\simeq(T(\Km_{3}(B)),\CC\omega_{\Km_{3}(B)}).
\]
\end{thm}

\begin{proof}
By Proposition \ref{prop:T(X)-unique-in-genus}, the lattice $T(X)$
is uniquely determined by the lattice $T(A)$ and conversely, in particular,
the lattice $T(X)$ is isometric to $T(X')$ (where $X'=\Km_{3}(B)$)
if and only if $T(A)$ is isometric to $T(B)$. Suppose that one has
an isomorphism of Hodge structures $(T(A),\CC\o_{A})\simeq(T(B),\CC\o_{B})$,
then $T(A)\simeq T(B)$ and from the hypothesis, there exists in $T(B)(3)\otimes\QQ$
a unique over-lattice $T'$ isometric to $T(X')$. The Hodge structure
$(T(X),\CC\omega_{X})$ is then uniquely determined by $(T,\CC\omega_{A})$:
necessarily $(T(X),\CC\omega_{X})\simeq(T,\CC\omega_{A})$, and also
$(T(X'),\CC\omega_{X'})\simeq(T',\CC\omega_{A})$, therefore the Hodge
structures $(T(X),\CC\o_{X})$ and $(T(X'),\CC\o_{X'})$ are isomorphic.

Conversely, suppose that there exists an isomorphism 
\[
\phi:(T(X),\CC\o_{X})\to(T(X'),\CC\o_{X'})
\]
of Hodge structures. Since $T$ contains a unique sub-lattice isometric
to $T(A)(3)$ and $T(X)\simeq T\simeq T(X')$, the sub-lattice $T(A)(3)$
is sent by $\phi$ to $T(B)(3)$ and there is a Hodge isomorphism
$(T(A),\CC\o_{A})\simeq(T(B),\CC\o_{B}).$ 
\end{proof}
Let us therefore study the even (as always) over-lattices $T$ of
$T(A)(3)$ such that $T(A)(3)$ has index $3$ in $T$. We consider
these over-lattices contained in $T(A)(3)\otimes\QQ$. The quadratic
form on the discriminant group $\text{A}_{T(A)(3)}$ is given by 
\[
Q=\left(\begin{array}{ccc}
-\tfrac{1}{6k} & 0 & 0\\
0 & \tfrac{2}{3} & -\tfrac{1}{3}\\
0 & -\tfrac{1}{3} & \tfrac{2}{9}
\end{array}\right).
\]
In the following list of $13$ elements of the form $(v;s_{q})\in\text{A}_{T(A)(3)}\times\QQ/2\ZZ$:
\begin{equation}
\begin{alignedat}{1} & (0,0,\frac{1}{3};0)\\
 & (0,\frac{1}{3},0;\tfrac{2}{3}),(0,\frac{1}{3},\frac{1}{3};\tfrac{2}{3}),(0,\frac{1}{3},\frac{2}{3};\tfrac{14}{3}),\\
 & (\frac{1}{3},0,0;-\tfrac{2}{3}k),(\frac{1}{3},0,\frac{1}{3};-\tfrac{2}{3}k),(\frac{1}{3},0,\frac{2}{3};-\tfrac{2}{3}k)\\
 & (\frac{1}{3},\frac{1}{3},0;-\tfrac{2}{3}k+\tfrac{2}{3}),(\frac{1}{3},\frac{1}{3},\frac{1}{3};-\tfrac{2}{3}k+\tfrac{2}{3}),(\frac{1}{3},\frac{1}{3},\frac{2}{3};-\tfrac{2}{3}k+\tfrac{14}{3}),\\
 & (\frac{1}{3},\frac{2}{3},0;-\tfrac{2}{3}k+\tfrac{8}{3}),(\frac{1}{3},\frac{2}{3},\frac{1}{3};-\tfrac{2}{3}k+\tfrac{2}{3}),(\frac{1}{3},\frac{2}{3},\frac{2}{3};-\tfrac{2}{3}k+\tfrac{8}{3}),
\end{alignedat}
\label{eq:liste}
\end{equation}
are the $13$ generators 
\[
v=(a_{1},a_{2},a_{3})\in\text{A}_{T(A)(3)}
\]
of the $13$ order $3$ sub-groups $H_{v}$ of $\text{A}_{T(A)(3)}\simeq\ZZ/6k\ZZ\times\ZZ/3\ZZ\times\ZZ/9\ZZ$,
and their square $s_{q}=vQ\,^{t}v\in\QQ/2\ZZ$. We denote by $T_{v}$
the pull-back in $T(A)(3)\otimes\QQ$ of the group $H_{v}\subset\text{A}_{T(A)(3)}$.
We remark that the groups generated by the elements 
\[
(0,\frac{1}{3},0;\tfrac{2}{3}),(0,\frac{1}{3},\frac{1}{3};\tfrac{2}{3}),(0,\frac{1}{3},\frac{2}{3};\tfrac{14}{3})
\]
in the list \eqref{eq:liste} are not isotropic sub-groups, thus their
pull-back to $T(A)(3)\otimes\QQ$ are not over-lattices of $T(A)(3)$.
For any $k$, the case $v_{0}=(0,0,\frac{1}{3})$ gives an index $3$
over-lattice isometric to $T(X)$.

Let $N^{over}$ be the number of over-lattices of $T(A)(3)$ in $T(A)(3)\otimes\QQ$
that are isomorphic to $T(X)$. From the above discussion on case
$v_{0}=(0,0,\frac{1}{3})$, we get that $N^{over}\geq1$. The aim
of the next two Sections is to prove the following result: 
\begin{thm}
\label{thm:NumberOfTX}We have\vspace{1mm}
\\
\begin{tabular}{|c|c|c|c|c|}
\hline 
 & $k=1\text{ or }2\mod3$  & $k=0\mod9$  & $k=3\mod9$  & $k=6\mod9$\tabularnewline
\hline 
$N^{over}$  & $1$  & $3$  & $1$  & $2$\tabularnewline
\hline 
$w$  & $(0,0,\frac{1}{3})$  & $\begin{array}{c}
(0,0,\frac{1}{3})\\
(\frac{1}{3},0,\frac{1}{3})\\
(\frac{1}{3},0,\frac{2}{3})
\end{array}$  & $(0,0,\frac{1}{3})$  & $\begin{array}{c}
(0,0,\frac{1}{3})\\
(\frac{1}{3},0,0)
\end{array}$\tabularnewline
\hline 
\end{tabular}\\
 where in the last line are the elements $w$ such that $T_{w}$ is
isometric to $T(X)$. 
\end{thm}

For $k=2\mod3$, we remark that the element $v=(0,0,\frac{1}{3})$
in the list \eqref{eq:liste} is the only possibility in order for
$T_{v}$ to be an over-lattice of $T(A)(3)$. Let us check the possibilities
from the list \eqref{eq:liste} according to the remaining cases $k=0$
or $1\mod3$.

\section{Case $k=0\protect\mod3$}

Let us prove Theorem \ref{thm:NumberOfTX} when $k=0\mod3$. In this
case $k=3k'$ for $k'\in\ZZ$, the possibilities different from $v_{0}=(0,0,\frac{1}{3};0)$
are 
\[
\begin{array}{c}
(\frac{1}{3},0,0;-\tfrac{2}{3}k),(\frac{1}{3},0,\frac{1}{3};-\tfrac{2}{3}k),(\frac{1}{3},0,\frac{2}{3};-\tfrac{2}{3}k).\end{array}
\]

\subsection{Case $k=3k'$ and $v=(\frac{1}{3},0,0)$}

Let us study the case 
\[
v=(\frac{1}{3},0,0)\in\text{A}_{T(A)(3)}.
\]
We have that $v^{2}=-\tfrac{2}{3}k$, and this is $0$ in $\QQ/2\ZZ$
since $k=3k'$ by assumption. Then the Gram matrix in some basis of
the over-lattice $T_{v}$ associated to $H_{v}$ is 
\[
\left(\begin{array}{ccc}
-2k' & 0 & 0\\
0 & 6 & 9\\
0 & 9 & 18
\end{array}\right).
\]
The lattice $T_{v}$ has discriminant group isomorphic to $\ZZ2k'\ZZ\times\ZZ/3\ZZ\times\ZZ/9\ZZ$,
generated by the columns of the matrix 
\[
\left(\begin{array}{ccc}
-\frac{1}{2k'} & 0 & 0\\
0 & \frac{2}{3} & \frac{-1}{3}\\
0 & -\frac{1}{3} & \frac{2}{9}
\end{array}\right).
\]

\begin{prop}
\label{prop:6mod9}The lattice $T_{v}$ is isometric to $T(X)$ if
and only if $k=6\mod9$. 
\end{prop}

\begin{proof}
The elements 
\[
(\frac{1}{2k'},0,0),(0,\tfrac{2}{3},-\tfrac{1}{3}),(0,0,\tfrac{1}{9})
\]
generate $\text{A}_{T_{v}}\simeq\ZZ/2k'\ZZ\times\ZZ/3\ZZ\times\ZZ/9\ZZ$
and have intersection matrix $\text{Diag}(-\frac{1}{2k'},\frac{2}{3},\frac{2}{9})$.

The lattice $T(X)$ has discriminant group isomorphic to 
\[
\ZZ/6k\ZZ\times\ZZ/3\ZZ=\ZZ/18k'\ZZ\times\ZZ/3\ZZ,
\]
thus if $3|k'$, $T_{v}$ is not isometric to $T(X)$. If $3\not|k'$,
then 
\[
\text{A}_{T(X)}\simeq\ZZ/2k'\ZZ\times\ZZ/3\ZZ\times\ZZ/9\ZZ\simeq\text{A}_{T_{v}}
\]
and we must compare the quadratic forms. The quadratic form on $T(X)$
is $\text{Diag}(-\frac{1}{18k'},\frac{2}{3})$. Since $3\not|k'$,
one has 
\[
(-\frac{1}{18k'})=(-\frac{u}{2k'})+(-\frac{v}{9}),
\]
where $u\in(\ZZ/2k'\ZZ)^{*}$ is such that $9u=1\mod2k'$ and $v\in(\ZZ/9\ZZ)^{*}$
is such that $2kv=1\mod9$. Since $9u=1\mod2k'$, there exists $u'\in(\ZZ/2k'\ZZ)^{*}$
such that $u'^{2}=u$, thus $(-\frac{u}{2k'})$ is isometric to $(-\frac{1}{2k'})$.
Suppose that $k'=1\mod3$, then $-(2k')^{-1}=1,4\text{\,or }7\mod9$,
and since $4=2^{2}\mod9$, $7=5^{2}\mod9$, we obtain that $(-\frac{v}{9})=(\frac{4}{9})$.
Suppose that $k'=2\mod3$, then $-(2k')^{-1}=2,5\text{\,or }8\mod9$,
and since $5=2\cdot5^{2}\mod9$, $8=2\cdot2^{2}\mod9$, we obtain
that $(-\frac{v}{9})=(\frac{2}{9})$. Therefore, the quadratic form
on $\text{A}_{T(X)}$ is isometric to 
\begin{equation}
\begin{array}{c}
(-\frac{1}{2k'})+(\frac{4}{9})+(\frac{2}{3})\text{ if\,}k'=1\mod3\,\\
(-\frac{1}{2k'})+(\frac{2}{9})+(\frac{2}{3})\text{ if\,}k'=2\mod3.
\end{array}\label{eq:no3onATX}
\end{equation}
Thus $T(X)$ is isometric to $T_{v}$ if and only if $k=6\mod9$.
In that case $T(X)$ and $T_{v}$ are in the same genus, but we know
from proposition \ref{prop:T(X)-unique-in-genus} that $T(X)$ is
unique in its genus, therefore $T(X)$ and $T_{v}$ are isometric. 
\end{proof}

\subsection{\label{subsec:Case(1/3,0,1/3)}Case $k=3k'$ and $v=(\frac{1}{3},0,\frac{1}{3})$}

Let $T_{v}$ be the over-lattice associated to $v=(\frac{1}{3},0,\frac{1}{3})$.
There is basis in which the Gram matrix of $T=T_{v}$ is 
\[
\left(\begin{array}{ccc}
-2k'+2 & 0 & 3\\
0 & 6 & 9\\
3 & 9 & 18
\end{array}\right).
\]
The discriminant group $\text{A}_{T}$ of $T$ has order $54k'$ ;
it is generated by the columns $c_{1},c_{2},c_{3}$ of

\[
\left(\begin{array}{ccc}
-\frac{1}{2k'} & 0 & \frac{1}{3k'}\\
-\frac{1}{2k'} & \frac{2}{3} & \frac{1}{3}(\frac{1}{k'}-1)\\
\frac{1}{3k'} & -\frac{1}{3} & \frac{2}{9}(1-\frac{1}{k'})
\end{array}\right).
\]
Making the substitution $c_{2}\to2c_{2}+2k'c_{1}$, (which is possible
since $c_{2}$ has order $3$), gives the generators 
\begin{equation}
\left(\begin{array}{ccc}
-\frac{1}{2k'} & 0 & \frac{1}{3k'}\\
-\frac{1}{2k'} & \frac{1}{3} & \frac{1-k'}{3k'}\\
\frac{1}{3k'} & 0 & \frac{2}{9}(\frac{k'-1}{k'})
\end{array}\right)\label{eq:no1}
\end{equation}
with intersection matrix 
\[
\left(\begin{array}{ccc}
\frac{-1}{2k'} & 0 & \frac{1}{3k'}\\
0 & \frac{2}{3} & 0\\
\frac{1}{3k'} & 0 & \frac{2(k'-1)}{9k'}
\end{array}\right),
\]
thus $T$ is isometric to 
\begin{equation}
\left(\frac{2}{3}\right)+\left(\begin{array}{cc}
\frac{-1}{2k'} & \frac{1}{3k'}\\
\frac{1}{3k'} & \frac{2(k'-1)}{9k'}
\end{array}\right).\label{eq:no4}
\end{equation}
We will also use the transformation $c_{3}\to c_{3}+c_{2}$ in Equation
\ref{eq:no1}; then one gets the generators $e_{1},e_{2},e_{3}$ 
\begin{equation}
\left(\begin{array}{ccc}
-\frac{1}{2k'} & 0 & \frac{1}{3k'}\\
-\frac{1}{2k'} & \frac{1}{3} & \frac{1}{3k'}\\
\frac{1}{3k'} & 0 & \frac{2}{9}(\frac{k'-1}{k'})
\end{array}\right)\label{eq:no2}
\end{equation}
with intersection matrix 
\[
\left(\begin{array}{ccc}
-\frac{1}{2k'} & 0 & \frac{1}{3k'}\\
0 & \frac{2}{3} & \frac{2}{3}\\
\frac{1}{3k'} & \frac{2}{3} & \frac{2(4k'-1)}{9k'}
\end{array}\right).
\]

\subsubsection{Sub-case $v=(\frac{1}{3},0,\frac{1}{3})$ and $k=0\protect\mod9$}

Suppose that $k=0\mod9$. Let us prove 
\begin{lem}
The lattice $T$ associated to $v=(\frac{1}{3},0,\frac{1}{3})$ is
isometric to $T(X)$ 
\end{lem}

\begin{proof}
Let be $k'\in\NN$ such that $k=3k'$ and let us define $k''$ by
$k'=3^{\a}k''$ with $\a\geq1$ and $k''$ prime to $3$. Let us study
the torsion quadratic module 
\[
Q=\left(\begin{array}{cc}
\frac{-1}{2k'} & \frac{1}{3k'}\\
\frac{1}{3k'} & \frac{2(k'-1)}{9k'}
\end{array}\right)
\]
in Equation \ref{eq:no4}. Let $\text{A}_{T}^{\perp3}$ be the subgroup
of elements that have prime to $3$ order in the finite abelian group
$\text{A}_{T}$. We can then define $v_{2}'=v_{2}+\frac{2}{3}v_{1}$,
where $v_{1},v_{2}$ are the base vectors. Let $P=\left(\begin{array}{cc}
1 & \text{\ensuremath{\frac{2}{3}}}\\
0 & 1
\end{array}\right)$ be the basis change matrix; one has 
\[
^{t}PQP=\left(\begin{array}{cc}
\frac{-1}{2k'} & 0\\
0 & \frac{2}{9}
\end{array}\right)\simeq(\frac{-1}{2k'})+(\frac{2}{9}),
\]
where since we consider the sub-group $\text{A}_{T}^{\perp3}$, the
form $(\frac{2}{9})$ is the zero form, and the quadratic form on
$\text{A}_{T}^{\perp3}$ is therefore $(-\frac{1}{2k'})$. For the
$3$-torsion sub-group $\text{A}_{T}(3)$, let us consider $v_{1}'=v_{1}+\frac{3}{2-2k'}v_{2}$
(one has $2-2k'=2\mod3$). Under the basis change by $P'=\left(\begin{array}{cc}
1 & 0\\
\frac{3}{2-2k'} & 1
\end{array}\right)$, the quadratic form on $\text{A}_{T}(3)$ is isometric to 
\[
\left(\begin{array}{cc}
\frac{-1}{2(k'-1)} & 0\\
0 & \frac{2}{9}(1-\frac{1}{k'})
\end{array}\right)\simeq\left(\frac{-1}{2(k'-1)}\right)+\left(\frac{2}{9}(1-\frac{1}{k'})\right)
\]
and since we consider here the sub-group $\text{A}_{T}(3)$ and $k'-1=-1\mod3$,
the form $\left(\frac{-1}{2(k'-1)}\right)$ is zero, and the quadratic
form on $\text{A}_{T}(3)$ is isometric to 
\[
\left(\frac{2}{9}(1-\frac{1}{k'})\right)=\left(\frac{2(k'-1)/k''}{3^{\a+2}}\right).
\]
We thus conclude that the quadratic form on $T$ is 
\[
\begin{array}{c}
(-\frac{1}{2k'})\text{ for elements in A}_{T}^{\perp3},\\
\left(\frac{2(k'-1)}{9k'}\right)+(\frac{2}{3})\text{\text{ for elements in A}}_{T}(3).
\end{array}
\]
Now the quadratic form on $T(X)$ is 
\[
(-\frac{1}{6k})+(\frac{2}{3}).
\]
The two quadratic forms are equal if they are equal at any prime.
At prime $3$, they are equal if and only if 
\[
(-\frac{1/(2k'')}{3^{\a+2}})=(\frac{2(k'-1)/k''}{3^{\a+2}}),
\]
this is the case if and only if 
\[
\frac{2(k'-1)/k''}{-1/(2k'')}=4(1-3^{\a}k'')
\]
is a square modulo $3^{\a+2}$, this is equivalent to $1-3^{\a}k''\in3^{\a+2}$
is a square. Let $v=1-3^{\a}k''/2\in\ZZ/3^{\a+2}\ZZ$, then $v^{2}=1-3^{\a}k''+3^{2\a}k''/4$
and therefore, when $\a\geq2$, $v^{2}=1-3^{\a}k''$ is a square.
One can check that this is also true for $\a=1$ by a direct computation.
Thus we obtain that, at prime $3$, the two quadratic forms are equal.

Now for the part coprime to $3$, we must compare $(-\frac{1}{2k'})=(-\frac{1/3^{\a}}{2k''})$
with $(-\frac{1}{6k})=(-\frac{1/3^{\a+2}}{2k''})$: these forms are
isometric since they differ by a square.

We thus proved that the two lattices $T,\,T(X)$ are in the same genus.
Since the genus contains a unique element by Proposition \ref{prop:T(X)-unique-in-genus},
the two lattices are isomorphic. 
\end{proof}

\subsubsection{Sub-case $v=(\frac{1}{3},0,\frac{1}{3})$ and $k=3\protect\mod9$}

Suppose that $k=3k'$ with $k'=1\mod3$, and let us prove 
\begin{lem}
The lattice $T$ associated to $v=(\frac{1}{3},0,\frac{1}{3})$ is
not isometric to $T(X)$ 
\end{lem}

\begin{proof}
Using Equation \ref{eq:no1}, the discriminant group is generated
by the columns of 
\[
\left(\begin{array}{ccc}
-\frac{1}{2k'} & 0 & \frac{1}{3k'}\\
-\frac{1}{2k'} & \frac{1}{3} & -\frac{k''}{k'}\\
\frac{1}{3k'} & 0 & \frac{2k''}{3k'}
\end{array}\right)
\]
(where $k'=3k''+1$) and we remark that 
\[
c_{3}-2k''c_{1}=(\frac{1}{3},0,0)
\]
thus the discriminant group $\text{A}_{T}$ is generated by the columns
of 
\[
\left(\begin{array}{ccc}
-\frac{1}{2k'} & 0 & \frac{1}{3}\\
-\frac{1}{2k'} & \frac{1}{3} & 0\\
\frac{1}{3k'} & 0 & 0
\end{array}\right)
\]
which shows that, when $k'=1\mod3$, the discriminant group of $T$
is 
\[
\text{A}_{T}\simeq(\ZZ/3\ZZ)^{3}\times\ZZ/2k'\ZZ.
\]
It is not isomorphic to $\ZZ/18k'\ZZ\times\ZZ/3\ZZ$, therefore $T$
is not isometric to $T(X)$. 
\end{proof}

\subsubsection{Sub-case $v=(\frac{1}{3},0,\frac{1}{3})$ and $k=6\protect\mod9$}

Suppose that $k=3k'$ with $k'=2\mod3$. Let us prove 
\begin{lem}
The lattice $T$ associated to $v=(\frac{1}{3},0,\frac{1}{3})$ is
not isometric to $T(X)$ 
\end{lem}

\begin{proof}
One can check that, as abstract groups, the discriminant groups $\text{A}_{T(X)}$
and $\text{A}_{T}$ are isomorphic. Let us study the $3$-torsion
part, denoted by $\text{A}_{T}(3)$, of $\text{A}_{T}$ and the quadratic
form restricted to that part. Since $3\not|k'$, from \ref{eq:no2},
the generators of $\text{A}_{T}(3)$ are the columns $c_{1},c_{2},c_{3}$
of the matrix 
\[
\left(\begin{array}{ccc}
0 & 0 & \frac{1}{3}\\
0 & \frac{1}{3} & \frac{1}{3}\\
\frac{4}{3} & 0 & \frac{2}{9}(k'-1)
\end{array}\right).
\]
Since $k'\neq1\mod3$, the element $c_{1}$ is a multiple of $c_{3}$
and the generators of $\text{A}_{T}(3)$ are $c_{2},c_{3}$, with
intersection matrix 
\[
\left(\begin{array}{cc}
\frac{2}{3} & \frac{2k'}{3}\\
\frac{2k'}{3} & \frac{2k'}{9}(4k'-1)
\end{array}\right).
\]
Then $c_{3}'=c_{3}-kc_{2}$ is such that $c_{3}'c_{2}=0$ and $c_{3}'^{2}=\frac{2}{9}k'(k'-1)$.
We obtain that $\text{A}_{T}(3)$ is isometric to $(\frac{2}{3})+(\frac{2k'(k'-1)}{9}).$
Using that $k'=2,5$ or $8\mod9$, it is easy to check that in any
case $(\frac{2k'(k'-1)}{9})\simeq(\frac{4}{9})$. We conclude that
$\text{A}_{T}(3)$ is isometric to $(\frac{2}{3})+(\frac{4}{9}).$
But we know from Equation \ref{eq:no3onATX} that when $k'=2\mod3$,
the $3$-torsion part of $\text{A}_{T(X)}$ is $(\frac{2}{3})+(\frac{2}{9})$.
Therefore $T$ is not isometric to $T(X)$ when $k'=2\mod3$. 
\end{proof}

\subsection{Case $k=3k'$ and $v=(\frac{1}{3},0,\frac{2}{3})$}

Let $v=(\frac{1}{3},0,\frac{2}{3})\in\text{A}_{T(A)(3)}$, $v'=(\frac{1}{3},0,\frac{1}{3})$
and let $T_{v},T_{v'}$ be the associated over-lattices. 
\begin{lem}
The lattices $T_{v}$ and $T_{v'}$ are isomorphic. The lattice $T_{v}$
is isometric to $T(X)$ if and only if $k=0\mod9$. 
\end{lem}

\begin{proof}
The matrix 
\begin{equation}
g=\left(\begin{array}{ccc}
1 & 0 & 0\\
0 & -1 & -3\\
0 & 1 & 2
\end{array}\right)\label{eq:g-Isometrie}
\end{equation}
is an order $6$ isometry of the lattice $T(A)(3)$ with Gram matrix
\[
Q_{T(A)(3)}=\left(\begin{array}{ccc}
-6k & 0 & 0\\
0 & 6 & 9\\
0 & 9 & 18
\end{array}\right)
\]
(one has $^{t}gQ_{T(A)(3)}g=Q_{T(A)(3)}$) and $g(\tfrac{1}{3}(v_{1}+v_{3}))=\tfrac{1}{3}(v_{1}+3v_{2}+2v_{3})$.
Its action on the discriminant group sends the group $H_{v'}$ generated
by the class of $v'=\tfrac{1}{3}(v_{1}+v_{3})$ in $\text{A}_{T(A)(3)}$
to the group $H_{v}$ generated by the class of $v=\tfrac{1}{3}(v_{1}+2v_{3})$
in $\text{A}_{T(A)(3)}$. Therefore, the element $g$ induces an isometry
between the over-lattices $T_{v'},T_{v}$ which are the pull-back
in $T(A)(3)\otimes\QQ$ of $H_{v'}$ and $H_{v}$. The lattice $T_{v}$
has therefore the same properties as the lattice $T_{v'}$ studied
in Section \ref{subsec:Case(1/3,0,1/3)}. 
\end{proof}

\section{Case $k=3k'+1$}

Let us prove Theorem \ref{thm:NumberOfTX} when $k=1\mod3$. When
$k=3k'+1$, the cases different from $(0,0,\frac{1}{3})$ are the
six generators $v$ in the following list 
\[
\begin{array}{c}
(\frac{1}{3},\frac{1}{3},0),(\frac{1}{3},\frac{1}{3},\frac{1}{3}),(\frac{1}{3},\frac{1}{3},\frac{2}{3}),\\
(\frac{1}{3},\frac{2}{3},0),(\frac{1}{3},\frac{2}{3},\frac{1}{3}),(\frac{1}{3},\frac{2}{3},\frac{2}{3}),
\end{array}
\]
of the six order $3$ isotropic groups $H_{v}$. 
\begin{lem}
Let $T$ be the over-lattice corresponding to the isotropic order
$3$ group $H_{v}=\left\langle v\right\rangle $, with $v$ in the
above list. The lattice $T$ is not isometric to $T(X)$. 
\end{lem}

\begin{proof}
The isometry $g$ in Equation \ref{eq:g-Isometrie} acts on the discriminant
group $\text{A}_{T(A)(3)}$ and the above $6$ elements form one orbit
for that action. It is therefore enough to study the over-lattice
$T$ corresponding to one of these elements, say $v=(\frac{1}{3},\frac{1}{3},0)$.
The Gram matrix of the corresponding over-lattice in some basis is
\[
Q_{T}=\left(\begin{array}{ccc}
-2k' & 1 & 0\\
1 & 6 & 9\\
0 & 9 & 18
\end{array}\right)
\]
the discriminant group has order $54k'+18=18k$ (here, $k=3k'+1$);
it is generated by the columns $c_{1},c_{2},c_{3}$ of 
\[
Q_{T}^{-1}=\left(\begin{array}{ccc}
-\frac{3}{2k} & \frac{1}{k} & -\frac{1}{2k}\\
\frac{1}{k} & \frac{2k'}{k} & -\frac{k'}{k}\\
-\frac{1}{2k} & -\frac{k'}{k} & \frac{4k-3}{18k}
\end{array}\right).
\]
In $\text{A}_{T}$, one has $c_{2}=2k'c_{1}$. Moreover, $c_{1}-3c_{3}=(0,0,-\frac{2}{3})$.
We observe that the column $2kc_{3}$ is $(0,0,\frac{4k-3}{9})$.
Since $4k-3$ is coprime to $9$, the group generated by $2kc_{3}$
contains $(0,0,-\frac{2}{3})$. We thus obtain that the discriminant
group is cyclic, generated by $c_{3}$: $T$ cannot be isometric to
$T(X)$. 
\end{proof}

\section{Preservation of $T(A)(3)$ into $T(X)$ under $O(T(X))$, proof of
the main Theorem in case $k\protect\neq0,6\protect\mod9$}

We recall that $T(X)$ is the lattice with Gram matrix 
\[
Q_{1}=\left(\begin{array}{ccc}
-6k & 0 & 0\\
0 & 6 & 3\\
0 & 3 & 2
\end{array}\right)
\]
in basis $\beta_{1}=(e_{1},e_{2},e_{3})$. Let $T_{1}\simeq T(A)(3)$
be the lattice generated by $e_{1},e_{2},3e_{3}$. Let us show that 
\begin{prop}
\label{prop:PreservationOfT(A)(3)}The orthogonal group $O(T(X))$
preserves $T_{1}$. 
\end{prop}

\begin{proof}
Let $g=(a_{ij})_{1\leq i,j\leq3}\in O(T(X))$; one has $^{t}gQ_{1}g=Q_{1}$.
The lattice $T_{1}$ is preserved by $g$ if and only if $ge_{1},ge_{2},3ge_{3}\in T_{1}$.
Since $T_{1}=\left\langle e_{1},e_{2},3e_{3}\right\rangle $, this
is the case if and only if the coefficients $a_{31},a_{32}$ are $0\mod3$.
The relation $^{t}gQ_{1}g=Q_{1}$ implies that mod $3$, one has 
\[
\begin{array}{ccc}
^{t}gQ_{1}g= & \left(\begin{array}{ccc}
* & * & 2a_{31}a_{33}\\
* & * & 2a_{32}a_{33}\\
* & * & 2a_{33}^{2}
\end{array}\right) & =\left(\begin{array}{ccc}
0 & 0 & 0\\
0 & 0 & 0\\
0 & 0 & 2
\end{array}\right)\end{array},
\]
thus $a_{33}=1\text{ or }2\mod3$ and $a_{31}=a_{32}=0\mod3$, which
implies the result. 
\end{proof}
\begin{comment}
Let $\iota:T_{1}\hookrightarrow T(X)$ be an embedding and $T_{2}$
be the image of $\iota$. 
\begin{prop}
Suppose that $k=1\text{ or }2\mod3$ or $k=3\mod9$. Then $T_{1}=T_{2}$,
in other words: $T(X)$ contains a unique sub-lattice isomorphic to
$T(A)(3)$. 
\end{prop}

\begin{proof}
By Theorem \ref{thm:NumberOfTX}, $T(X)$ is up to isometries the
unique over-lattice containing $T_{1}$. By Proposition \ref{prop:PreservationOfT(A)(3)},
every isometry of $T(X)$ preserves $T_{1}$, therefore $T_{1}$ is
the unique sub-lattice isometric to $T(A)(3)$. 
\end{proof}
\end{comment}
Let $T_{2}$ be a lattice with Gram matrix 
\[
Q_{2}=\left(\begin{array}{ccc}
-6k & 0 & 0\\
0 & 6 & 9\\
0 & 9 & 18
\end{array}\right)
\]
in some basis $v_{1},v_{2},v_{3}$; it is isomorphic to $T_{1}$.
Let $\iota:T_{2}\hookrightarrow T(X)$ be an embedding of lattices.
Let us identify $T_{2}$ with its image in $T(X)$ trough the embedding
$\iota$. 
\begin{prop}
Suppose that $k\neq0\text{ and }6\mod9$. Then $T_{1}=T_{2}$, in
other words: $T(X)$ contains a unique sub-lattice isomorphic to $T(A)(3)$. 
\end{prop}

\begin{proof}
By Theorem \ref{thm:NumberOfTX}, the hypothesis on $k$ implies that
the over-lattice $T(X)$ of $T_{2}$ is obtained by $\frac{1}{3}v_{3}\in T(X)$.
Then $\beta_{2}=(v_{1},v_{2},\frac{1}{3}v_{3})$ is a basis of $T(X)$.
These vectors have intersection matrix $Q_{2}$ equal to $Q_{1}$.
Let $P$ be the base-change matrix between these two basis, one has
\[
Q_{1}={}^{t}PQ_{2}P
\]
and since $Q_{1}=Q_{2}$, the matrix $P$ sending the base $\beta_{1}$
to the base $\beta_{2}$ defines an element of $O(T(X))$. Since $O(T(X))$
preserves $T_{1}$, the vectors $v_{1},v_{2},v_{3}$ are in $T_{1}$,
thus $T_{1}=T_{2}$. 
\end{proof}
\begin{cor}
When $k\neq0\text{ and }6\,\mod9$, the hypothesis of Theorem \ref{thm:Main3}
are satisfied, and therefore it proves Theorem \ref{thm:Main1} in
these cases. 
\end{cor}

\section{Cases $L_{X}^{2}=6k$ with $k=0$ or $6\protect\mod9$}

\subsection{On the implication (II) $\Rightarrow$ (I) in cases $k=0$ or $6\protect\mod9$}

Suppose that $k=0$ or $6\mod9$. Let $v_{1},v_{2},\dots$ be the
generators of the (distinct) isotropic groups $H_{1},H_{2},\dots$
of $\text{A}_{T(A)(3)}$ such that the corresponding over-lattices
$T_{1},T_{2},...$ are isometric to $T(X)$. Let $(B,G_{B})$ be another
abelian surface with an order $3$ symplectic automorphism group. 
\begin{prop}
\label{prop:ConverseImplic}Suppose that $k=0$ or $6\mod9$ and that
the over-lattices $T_{1},T_{2},\dots$ of $T(A)(3)$ are isomorphic.
If there exists an isomorphism of Hodge structures 
\[
\psi:(T(\Km_{3}(A)),\CC\o_{\Km_{3}(A)})\to(T(\Km_{3}(B)),\CC\o_{\Km_{3}(B)})
\]
then there exists an isomorphism of Hodge structures 
\[
(T(A),\CC\o_{A})\simeq(T(B),\CC\o_{B}).
\]
\end{prop}

\begin{proof}
Since the over-lattices $T_{1},T_{2},\dots$ are isomorphic, by Corollary
\ref{cor:Unique-embedd-up-to-isom}, there is a unique embedding of
$T(A)(3)$ in $T(\Km_{3}(A))$ up to isometries. Moreover, by Proposition
\ref{prop:PreservationOfT(A)(3)}, any isometries of $T(\Km_{3}(A))$
preserves the image of $T(A)(3)$, thus we recover uniquely $(T(A),\CC\o_{A})$
from $(T(\Km_{3}(A)),\CC\o_{\Km_{3}(A)})$. The isomorphism $\psi$
must send $T(A)(3)$ to $T(B)(3)$, and therefore it induces an isomorphism
of Hodge structures between $(T(A),\CC\o_{A})$ and $(T(B),\CC\o_{B})$. 
\end{proof}
The next sub-section shows that the over-lattices $T_{1},T_{2},\dots$
of $T(A)(3)$ are indeed isomorphic, so that the hypothesis of Proposition
\ref{prop:ConverseImplic} are satisfied, and that will prove the
implication (II) $\Rightarrow$ (I) of Theorem \ref{thm:Main1}.

\subsection{On the orthogonal group of $T(A)(3)$ and isomorphic over-lattices}

Recall that $T(A)(3)$ has a basis with Gram matrix 
\begin{equation}
\left(\begin{array}{ccc}
-6k & 0 & 0\\
0 & 6 & 9\\
0 & 9 & 18
\end{array}\right).\label{eq:nom1-1}
\end{equation}

\begin{prop}
\label{prop:T(A)(3)UniqueGenusOrth}The lattice $T(A)(3)$ is unique
in its genus and the map 
\[
O(T(A)(3))\to O(\text{A}_{T(A)(3)})
\]
is surjective. 
\end{prop}

\begin{proof}
The discriminant group $\mathrm{A}_{T(A)(3)}$ is (isomorphic to)
$\ZZ/9\ZZ\times\ZZ/3\ZZ\times\ZZ/6k\ZZ$. For a prime $p$, the length
$\ell_{p}$ of the $p$-torsion subgroup is $\ell_{3}=3$, $\ell_{p}=1$
for $p\neq3$ dividing $k$, otherwise $\ell_{p}=0$.

We use the book \cite{MirMor}, and we refer to \cite[Chapter VIII, Definition 7.4]{MirMor}
for the definition of regular or pseudo regular primes of quadratic
forms. By \cite[Chapter VIII, Lemma 7.7(1)]{MirMor}, the quadratic
form $Q$ is $2$-regular, by \cite[Chapter VIII, Lemma 7.6(3)]{MirMor}
$Q$ is $3$-pseudoregular, and by \cite[Chapter VIII, Lemma 7.6(1)]{MirMor}
it is $p$-regular for any prime $p\geq5$. One can therefore apply
\cite[Chapter VIII, Theorem 7.5 (4)]{MirMor} to conclude that the
genus of $T(A)(3)$ is $\{T(A)(3)\}$ and that the natural map $O(T(A)(3))\to O(\mathrm{A}{}_{T(A)(3)})$
is surjective. 
\end{proof}
If $L$ is a torsion quadratic module, there is a decomposition $L=\oplus_{p}L_{p}$
into $p$-torsion elements, and 
\[
O(L)=\prod_{p}O(L_{p}),
\]
where $O$ is the orthogonal group i.e. the group preserving the quadratic
form. The discriminant group $\text{A}_{T(A)(3)}$ has quadratic form
\[
Q=\left(\begin{array}{ccc}
-\tfrac{1}{6k} & 0 & 0\\
0 & \tfrac{2}{3} & -\tfrac{1}{3}\\
0 & -\tfrac{1}{3} & \tfrac{2}{9}
\end{array}\right).
\]
Let us prove the following result: 
\begin{prop}
\label{prop:6mod9TvTvpIsomeorphess}Suppose that $k=6\mod9$ and consider
$v_{0}=(0,0,\frac{1}{3})$ and $v_{1}=(\frac{1}{3},0,0)$. The associated
over-lattices $T_{v_{0}},T_{v_{1}}$ are isomorphic. 
\end{prop}

\begin{proof}
In order to prove that the over-lattices $T_{v_{0}},T_{v_{1}}$ are
isomorphic, one must show that there exists an element $O(T(A)(3))$
that acts on $\text{A}_{T(A)(3)}$ by sending $v_{0}$ to $v_{1}$.
But by Proposition \ref{prop:T(A)(3)UniqueGenusOrth}, the map $O(T(A)(3))\to O(\text{A}_{T(A)(3)})$
is surjective, therefore it is sufficient to prove that there exists
$\tau\in O(\text{A}_{T(A)(3)})$ sending $v_{0}$ to $v_{1}$. Since
we supposed that $k=6\mod9$, one has $\left(-\tfrac{1}{6k}\right)_{3}=\left(\frac{2}{9}\right)$,
thus on the $3$-torsion part, the quadratic form is 
\[
Q_{3}=\left(\begin{array}{ccc}
\frac{2}{9} & 0 & 0\\
0 & \tfrac{2}{3} & -\tfrac{1}{3}\\
0 & -\tfrac{1}{3} & \tfrac{2}{9}
\end{array}\right).
\]
The matrix 
\[
\tau=\left(\begin{array}{ccc}
0 & -6 & 1\\
0 & 1 & 0\\
1 & 6 & 0
\end{array}\right)
\]
defines an automorphism of the ($3$-torsion part of the) group $\text{A}_{T(A)(3)}=\ZZ/9\ZZ\times\ZZ/3\ZZ\times\ZZ/6k\ZZ$
and satisfies $^{t}\tau Q_{3}\tau-Q_{3}=\left(\begin{array}{ccc}
0 & 1 & 0\\
1 & 12 & -1\\
0 & -1 & 0
\end{array}\right)$, thus this is an element of the orthogonal group of $\text{A}_{T(A)(3)}$.
The transformation exchanges $(\frac{1}{9},0,0)$ and $(0,0,\frac{1}{9})$,
thus it also exchanges $v_{1}=(\frac{1}{3},0,0)$ and $v_{0}=(0,0,\frac{1}{3})$. 
\end{proof}
\begin{prop}
\label{prop:0mod9TvTvpIsomeorphess}Suppose that $k=0\mod9$. Let
be $v_{0}=(0,0,\frac{1}{3})$, $w_{1}=(\frac{1}{3},0,\frac{1}{3})$
and $w_{2}=(\frac{1}{3},0,\frac{2}{3})$, the elements of $\text{A}_{T(A)(3)}$.
The associated over-lattices $T_{v_{0}},T_{w_{1}},T_{w_{2}}$ are
isomorphic. 
\end{prop}

\begin{proof}
The isometry in Equation \eqref{eq:g-Isometrie} of $T(A)(3)$ sends
$w_{1}$ to $w_{2}$. By Proposition \ref{prop:T(A)(3)UniqueGenusOrth},
the map $O(T(A)(3))\to O(\text{A}_{T(A)(3)})$ is onto, therefore,
in order to prove that the over-lattices $T_{v_{0}},T_{w_{1}},T_{w_{2}}$
are isomorphic, it is sufficient to find an element $\tau\in O(\text{A}_{T(A)(3)})$
sending $v_{0}$ to an element in $\{w_{1},2w_{1},w_{2},2w_{2}\}$.
The group $\text{A}_{T(A)(3)}$ is isomorphic to 
\[
\ZZ/6k\ZZ\times\ZZ/3\ZZ\times\ZZ/9\ZZ
\]
through the map $(\frac{a}{6k},\frac{b}{3},\frac{c}{9})\to(a,b,c)$
and we will often identify $(\frac{a}{6k},\frac{b}{3},\frac{c}{9})$
with $(a,b,c)$. The discriminant group $\text{A}_{T(A)(3)}$ has
quadratic form 
\[
Q=\left(\begin{array}{ccc}
-\tfrac{1}{6k} & 0 & 0\\
0 & \tfrac{2}{3} & -\tfrac{1}{3}\\
0 & -\tfrac{1}{3} & \tfrac{2}{9}
\end{array}\right).
\]
At prime $3$, the quadratic form is 
\[
Q_{3}=\left(\begin{array}{ccc}
\tfrac{u}{3^{a+1}} & 0 & 0\\
0 & \tfrac{2}{3} & -\tfrac{1}{3}\\
0 & -\tfrac{1}{3} & \tfrac{2}{9}
\end{array}\right),
\]
where $k=3^{a}t$, for $a\geq2$, $t$ coprime to $3$ and $u$ is
even such that $u(-2t)=1\mod3^{a+1}$. The torsion quadratic module
$\left(\frac{u}{3^{a+1}}\right)$ is isometric to $\left(\frac{2}{3^{a+1}}\right)$
or $\left(\frac{4}{3^{a+1}}\right)$, and the isometry sends the order
$3$ element element $(3^{a},0,0)$ to $(w3^{a},0,0)$ with $w\in\{1,2\}$.
It remains therefore to study the following two possibilities: \\
 $\bullet$ Suppose that $u=2$ and $a\geq3$. The matrix 
\[
\tau_{a}=\left(\begin{array}{ccc}
2\cdot3^{a-1}-1 & -14\cdot3^{a} & 8\cdot3^{a-1}\\
0 & -4 & 1\\
7 & -48 & 11
\end{array}\right).
\]
define an endomorphism of the group 
\[
\text{A}_{T(A)(3)}=\ZZ/3^{a+1}\ZZ\times\ZZ/3\ZZ\times\ZZ/9\ZZ
\]
(for example, the $(3,2)$ entry is the well-defined map $\ZZ/3\ZZ\to\ZZ/9\ZZ$,
$\bar{x}\to\overline{-48x}$). The matrix 
\[
\left(\begin{array}{ccc}
2\cdot3^{a-1}-1 & -14\cdot3^{a} & 7\cdot3^{a-1}\\
3 & -10 & 2\\
20 & -87 & 17
\end{array}\right)
\]
also define an endomorphism and one can check that its product with
$\tau_{a}$ acts by the identity on $\text{A}_{T(A)(3)}$, thus $\tau_{a}$
is an automorphism of $\text{A}_{T(A)(3)}$. One computes that 
\[
^{t}\tau_{a}Q_{3}\tau_{a}-Q_{3}=\left(\begin{array}{ccc}
10+8\cdot3^{a-3} & -56\cdot(3^{a-2}+1) & 32\cdot3^{a-3}+13\\
-56\cdot(3^{a-2}+1) & 392(3^{a-1}+1)+2 & -224\cdot3^{a-2}-89\\
32\cdot3^{a-3}+13 & -224\cdot3^{a-2}-89 & 128\cdot3^{a-3}+20
\end{array}\right)
\]
therefore if $a\geq3$, the entries of the above matrix are integers
and the diagonal entries are even numbers, so that $\tau_{a}$ preserves
the quadratic form of $\text{A}_{T(A)(3)}$: it is an element of $O(\text{A}_{T(A)(3)}).$
One has $\tau_{a}(0,0,3)=(8\cdot3^{a},3,33)=(2\cdot3^{a},0,6)$, therefore
the image by $\tau_{a}$ of the isotropic group generated by $v_{0}=(0,0,3)$
is the isotropic group generated by $w_{1}=(3^{a},0,3)$: that implies
that the two associated over-lattices are isomorphic. \\
 In case $a=2$, it is not difficult to check that the matrix $\tau_{2}=\left(\begin{array}{ccc}
5 & -18 & 3\\
2 & -10 & 2\\
17 & -69 & 14
\end{array}\right)$ gives an isometry of $\text{A}_{T(A)(3)}$ with the same properties
as above when $a\geq3$.\\
 $\bullet$ Suppose that $u=4$ and $a\geq3$. The matrix 
\[
\t_{a}=\left(\begin{array}{ccc}
4\cdot3^{a-1}-1 & -14\cdot3^{a} & 8\cdot3^{a-1}\\
1 & -10 & 2\\
13 & -78 & 16
\end{array}\right).
\]
define an automorphism of the group $\text{A}_{T(A)(3)}$. One has
\[
^{t}\t_{a}Q_{3}\t_{a}-Q_{3}=\left(\begin{array}{ccc}
26+64\cdot3^{a-3} & -224\cdot3^{a-2}-144 & 128\cdot3^{a-3}+30\\
-224\cdot3^{a-2}-144 & 784\cdot3^{a-1}+898 & -448\cdot3^{a-2}-185\\
128\cdot3^{a-3}+30 & -448\cdot3^{a-2}-185 & 38+256\cdot3^{a-3}
\end{array}\right),
\]
so that $\t_{a}$ is an element of $O(\text{A}_{T(A)(3)})$ when $a\geq3$.
Since 
\[
\t_{a}(0,0,3)=(8\cdot3^{a},6,48)=(2\cdot3^{a},0,3),
\]
the image by $\t_{a}$ of the isotropic group generated by $v_{0}=(0,0,3)$
is the isotropic group generated by $2w_{2}=(2\cdot3^{a},0,3)$: that
implies that the two associated over-lattices are isomorphic. \\
 When $a=2$, it is not difficult to check that the matrix $\t_{2}=\left(\begin{array}{ccc}
2 & -126 & 24\\
0 & -4 & 1\\
5 & -66 & 14
\end{array}\right)$ gives an isometry with the same properties as above. 
\end{proof}
\begin{cor}
\label{cor:23}Suppose that $k=3k'$, with $k'=0\text{ or 2}\mod3$.
Suppose that there is an isomorphism of Hodge structures 
\[
(T(\Km_{3}(A)),\CC\o_{\Km_{3}(A)})\simeq(T(\Km_{3}(B)),\CC\o_{\Km_{3}(B)}).
\]
Then there is an isomorphism of Hodge structures 
\[
(T(A),\CC\o_{A})\simeq(T(B),\CC\o_{B}).
\]
\end{cor}

\begin{proof}
In these cases, all the over-lattices that are isometric to $T(X)$
are isomorphic, thus we can apply Proposition \ref{prop:ConverseImplic}. 
\end{proof}

\subsection{About the implication (I) $\Rightarrow$ (II) }

Suppose that $k=0$ or $6\mod9$. Let $v_{1},v_{2},\dots$ be generators
of the isotropic groups $H_{1},H_{2},\dots$ of $\text{A}_{T(A)(3)}$
such that the corresponding over-lattices $T_{1},T_{2},...$ of $T(A)(3)$
are isometric to $T(X)$. We know from Propositions \ref{prop:6mod9TvTvpIsomeorphess}
and \ref{prop:0mod9TvTvpIsomeorphess} that these lattices $T_{1},T_{2},\dots$
are isomorphic. 
\begin{prop}
\label{prop:KummerNotIso}Let be $i\neq j$ and let $\omega$ defining
a Hodge structure on $T_{i}\otimes\QQ=T_{j}\otimes\QQ$. The Hodge
structures $(T_{i},\CC\o)$ and $(T_{j},\CC\o)$ are not isomorphic
for a general period $\o$. 
\end{prop}

\begin{proof}
Let us recall some facts about integral Hodge structures of K3 type,
for which a reference is \cite[Section 7.2.3]{Voisin}. A Hodge structure
on a rank $3$ lattice $T$ of signature $(2,1)$ is the data of a
point $\o\in\PP(T\otimes\CC)$ such that $w^{2}=0$ and $w\bar{w}>0$,
for $w\in\o=\CC w$. Then the space $T\otimes\CC$ decomposes as 
\[
T\otimes\CC=\CC w\oplus\CC t\oplus\CC\bar{w},
\]
where $\CC t$ is the orthogonal complement of the space $\CC w\oplus\CC\bar{w}$;
this is a real subspace: $\overline{\CC t}=\CC t$. The set of Hodge
structures is an (euclidian) open subset in the smooth quadric $Q_{T}\simeq\PP^{1}$
defined by $\o^{2}=0$. In fact, this is the complement of the real
axis in $\CC\subset\PP^{1}=Q_{T}$, in particular, it is biholomorphic
to $\mathcal{H}\cup\overline{\mathcal{H}}$, where $\mathcal{H}$
is the complex upper-plane (see \cite[Section 2.3]{Barth}). Two Hodge
structures $\o,\o'$ on $T$ are isomorphic if and only if there exists
an isometry $g\in O(T)$ such that 
\[
g_{\CC}(\o)=\o',
\]
where $g_{\CC}$ is the complexification of $g$ (an isometry of $T$
acts as an homography on $\mathcal{H}\cup\overline{\mathcal{H}}\subset Q_{T}$).

The fixed point set of a projective automorphism $g$ acting on $\PP(T\otimes\CC)$
is a union of linear subspaces, thus the stabilizer group (in the
projective automorphism group) of the smooth quadric $Q_{T}$ acts
faithfully on $Q_{T}$.

Since we suppose that the over-lattices $T_{i}$ and $T_{j}$ are
isomorphic, there exist isometries $h_{0}\in O(T(A)(3))$ and $h:T_{i}\to T_{j}$
such that the following diagram is commutative 
\[
\begin{array}{ccc}
T(A)(3) & \stackrel{}{\hookrightarrow} & T_{i}\\
\downarrow h_{0} &  & \downarrow h\\
T(A)(3) & \stackrel{}{\hookrightarrow} & T_{j}
\end{array}.
\]
We consider $T_{i}$ and $T_{j}$ contained in $T(A)(3)\otimes\CC$.
The complexification $h_{\CC}$ of the lattice isometry $h$ preserves
the quadric $Q=\{\o^{2}=0\}\subset\PP(T(A)(3)\otimes\CC)$ containing
the Hodge structures. That quadric is the same for $T(A)(3),T_{i}$
and $T_{j}$. For $\o\in Q$, the isometry $h$ induces an isomorphism
between the Hodge structures $(T_{i},\o)$ and $(T_{j},h_{\CC}(\o))$.

Suppose that for a general period $\o\in Q$, the Hodge structures
$(T_{i},\o)$ and $(T_{j},\o)$ are isomorphic. Then $(T_{j},\o)$
and $(T_{j},h_{\CC}(\o))$ are isomorphic: there exists an isometry
$g\in O(T_{j})$ such that $g_{\CC}(\o)=h_{\CC}(\o)$. Since $O(T_{j})$
is a countable set, there exists a $g\in O(T_{j})$ such that for
an infinite number of $\o$, one has $g_{\CC}(\o)=h_{\CC}(\o)$, and
therefore, in fact, $g_{\CC}(\o)=h_{\CC}(\o)$ for all $\o$, thus
$\forall\omega\in Q,\,\,h_{\CC}^{-1}g_{\CC}(\o)=\o.$ Since the projective
automorphism group preserving $Q$ acts faithfully on $Q$ , this
implies that $h=\pm g$, and $h$ is an isometry of $T_{j}$. This
is a contradiction, therefore for general $\o\in Q$, the Hodge structures
$(T_{j},\o)$ and $(T_{j},h_{\CC}(\o))$ are not isomorphic, and we
conclude that $(T_{j},\omega)$, $(T_{i},\omega)$ are not isomorphic. 
\end{proof}
Recall that from Propositions \ref{prop:0mod9TvTvpIsomeorphess} and
\ref{prop:6mod9TvTvpIsomeorphess}, for $k=0$ (respectively $6\mod9$),
we denote by $T_{1},T_{2},T_{3}$ (resp. $T_{1},T_{2}$) the over-lattices
of $T(A)(3)$ that are isometric to $T(X)$. Let be $t=3$ if $k=0\mod9$
and $t=2$ if $k=6\mod9$. Let us fix a general period $\o$. The
Hodge structure $(T(A)(3),\o)$ induces Hodge structures 
\[
(T_{j},\o),j=\{1,\dots,t\}
\]
and by Proposition \ref{prop:KummerNotIso}, these Hodge structures
are not isomorphic for general $\o$. By the subjectivity of the period
map, we obtain 
\begin{cor}
Under the above notations, there exist generalized Kummer surfaces
surfaces $X_{1},\dots,X_{t}$ such that $(T(X_{s}),\CC\o_{X_{s}})\simeq(T_{s},\CC\o)$
for $s\in\{1,...,t\}$, with $X_{s}=\Km(A_{s},G_{s})$, such that
$A_{1},...,A_{t}$ are Fourier-Mukai partners but such that $X_{1},\dots,X_{t}$
are not isomorphic. 
\end{cor}

\begin{rem}
We remark that one has an isomorphism of $\QQ$-Hodge structures $(T(X_{i})\otimes\QQ,\CC\o)\simeq(T(X_{j})\otimes\QQ,\CC\o)$,
and according to \cite{Nikulin2} it is algebraic, i.e. is induced
by a correspondence between $X_{i}$ and $X_{j}$. 
\end{rem}

\vspace{2mm}

\noindent Xavier Roulleau\\
 Université d'Angers, \\
 CNRS, LAREMA, SFR MATHSTIC, \\
 F-49000 Angers, France

\noindent xavier.roulleau@univ-angers.fr
\begin{verbatim*}
https://math.univ-angers.fr/membre/roulleau-xavier/
\end{verbatim*}
\vspace{2mm}

\noindent Alessandra Sarti\\
 Université de Poitiers\\
 Laboratoire de Mathématiques et Applications,\\
 UMR 7348 du CNRS, \\
 TSA 61125 \\
 11 bd Marie et Pierre Curie, \\
 86073 Poitiers Cedex 9, France

\noindent Alessandra.Sarti@math.univ-poitiers.fr 
\begin{verbatim}
http://www-math.sp2mi.univ-poitiers.fr/~sarti/
\end{verbatim}

\end{document}